\documentclass{siamart250211}
\usepackage{amsmath,amssymb,amscd,amsfonts}
\usepackage{color}
\usepackage{graphicx}
\usepackage{mathtools}     
\usepackage{xfrac}
\usepackage{algorithm,algpseudocode}
\usepackage{tikz}
\usepackage{pgfplots}
\pgfplotsset{compat=1.18} 

\newcommand{\pyr}{\operatorname{pyr}}
\newcommand{\aff}{\operatorname{aff}}
\newcommand{\conv}{\operatorname{conv}}

\newcommand{\xb}{\mathbf{x}}
\newcommand{\yb}{\mathbf{y}}
\newcommand{\zb}{\mathbf{z}}
\newcommand{\vb}{\mathbf{v}}
\newcommand{\ub}{\mathbf{u}}
\newcommand{\wb}{\mathbf{w}}
\newcommand{\ab}{\mathbf{a}}

\newcommand{\lambdab}{{\boldsymbol{\lambda}}}
\newcommand{\xib}{{\boldsymbol{\xi}}}

\newcommand{\R}{\mathbb{R}}

\newcommand{\V}{\mathcal{V}}

\newcommand{\half}{\frac{1}{2}}

\newcommand{\abs}[1]{\left| #1 \right|}

\newcommand{\leaves}[1]{\mathcal{L}(#1)}

\newcommand{\faces}[2]{\mathcal{F}(#1,#2)}

\title{Space-Time Galerkin Boundary Element Method for the Wave Equation}
\author{Johannes Tausch\thanks{
    Southern Methodist University,Dallas,TX, USA
    (\email{tausch@smu.edu})} }

\headers{Galerkin Boundary Element Method for the Wave Equation}{J.Tausch}

\begin{document}

\maketitle

\begin{abstract}
  The space-time Galerkin discretization of retarded layer potentials leads to a linear system where the coefficients are expressed in terms of integrals over the ansatz and test elements. They require carefully designed quadrature schemes because the kernel is singular in the origin and is discontinuous across the hyperbolicity cone.  This paper introduces a new integration approach that leads to a scheme that converges exponentially with the number of quadrature points.  The key here is to consider the integration domain as a convex polytope and to devise a decomposition into the convex hulls of simpler polytopes which are parameterized such that the singularity and discontinuity occurs in a single variable.  The method is implemented for piecewise constant elements and tested on a scattering problem with known analytic solution.
\end{abstract}

\begin{keywords}
  Retarded potentials, Space-Time Galerkin Discretization, Convex polytopes, Numerical integration
\end{keywords}

\begin{AMS}
  65M38  
  52B11  
  65D30  
  35L05  
\end{AMS}

\section{Introduction}
Boundary integral equation methods are widely and successfully used for
the numerical solution of scattering problems.  In the case of the
wave equation the integration in Green's representation formula is over
the intersection of the hyperbolicity cone with the lateral boundary
of the space-time cylinder. One can roughly distinguish two
discretization approaches for these ``retarded potentials''.
The first is to construct a discretization method in frequency domain
and then transform back to time domain. This is the convolutional
quadrature method~\cite{lubich94,schanz-antes97,kress-sauter08}.  The
second is to discretize the space and time variables directly, either by
collocation\cite{davies-duncan04}, convolution spline method
\cite{davies-duncan14}, or space-time Galerkin methods.
For a more
comprehensive overview of this subject the reader is referred to the
survey~\cite{costabel04} and the monograph~\cite{sayas16}.

The study of retarded potentials in weak form 
was initiated in~\cite{bamberger-duong86}. An alternative
approach are the energetic variational
formulations~\cite{aimi-etal09,aimi-etal10}.  Further examples of papers that
develop space-time Galerkin methods and describe applications
to real-life problems are
\cite{banz-etal16,gimperlein-etal17,schneider-etal26}. We also refer
to the review \cite{joly-rodríguez17}.

The main focus of this paper is on the numerical evaluation of the
integrals in the matrix coefficients of the discretized linear
system. As will be demonstrated in Section~\ref{sec:galerkin} they
involve calculating integrals of the form
\begin{equation}\label{def:integral}
  I = \int_{P_x} \int_{P_y}
  \frac{k(\abs{\xb-\yb})}{\abs{\xb-\yb}} \Phi(\xb,\yb) \, d\yb d\xb
\end{equation}
where $P_x$, $P_y$ are the domains of the finite elements and
$\Phi(\xb,\yb)$ is a smooth function that represents the contributions
of the spatial shape functions. The kernel $k$ depends on the
choice of basis
functions in time. If standard piecewise polynomials are used, $k(r)$
is a polynomial in the interval $[0,\rho^*]$ and vanishes for $r>\rho^*$.

Evaluating \eqref{def:integral} with numerical quadrature is difficult
for two reasons: First, the integrand is singular when $P_x$ and $P_y$
have common vertices. The second difficulty is the limited regularity
of the integrand when the distance of $\xb$ and $\yb$ is near the
truncation parameter.

Integrals with $\rho^*=\infty$ arise in Galerkin boundary element
methods for elliptic equations. For the cases that $P_x$ and $P_y$
have common edges singularity removing
transformations have been developed that convert \eqref{def:integral} into a
number of smooth integrals over the four dimensional unit cube. The
transformed integrals can be computed efficiently with tensor product
Gauss Legendre rules, see~\cite{sauter-schwab11}.

To address the regularity issue for hyperbolic equations the papers
\cite{khoromskij-etal11, sauter-veit13} propose to use globally smooth
functions for the time discretization which in turn lead to a globally
smooth kernel. To handle the discontinuities with more standard
elements one usually treats \eqref{def:integral} as an iterated
integral.  To calculate the inner integral for a fixed $\xb$ the
points of $P_y$ are trimmed off that are farther from $\xb$ than
$\rho^*$. This requires some geometry calculations which can get
tedious as there are several cases of how $P_y$ intersects. The
greater problem is that the inner integral is in general not a
smooth function of $\xb$, which then makes its integration tricky.
Approaches along these lines have been discussed in a number of
previous investigations, see, e.g., \cite{aimi-etal13, ostermann10,
  polz-schanz19, polz-schanz21}.

In this paper we consider the domain in \eqref{def:integral} as the
Cartesian product of $P_x$ and $P_y$, which is a convex polytope in
$\R^6$. The idea is to employ the pyramidal decomposition algorithm
which was initially introduced in~\cite{tausch26} for kernel integrals
with $\rho^*=\infty$.  This algorithm will be modified such that the
decomposition of $P_x\times P_Y$ results in integrals that either
require no truncation or can be truncated easily in a single variable.

The outline of this paper is as follows. Section~\ref{sec:galerkin}
reviews how the wave equation can be solved using Galerkin
discretizations of retarded layer potentials. The following
section introduces the concepts from polytope
theory that will be used to describe the pyramidal decomposition
algorithm discussed in Section~\ref{sec:decomp}. Section~\ref{sec:analysis}
demonstrates that each piece generated in the decomposition can be
parameterized in a way such that the singularity and discontinuity
appears in a single variable. Finally,
Section~\ref{sec:implementation} gives more details about the
implementation and presents results obtained with a problem with
spherical geometry.

\section{Galerkin Discretization of Retarded Potentials}\label{sec:galerkin}
To keep the exposition simple we focus on the Dirichlet problem of the
wave equation posed in the exterior of a finite scattered $\Omega
\subset \R^3$. If $\Omega^c = \R^3 \setminus \bar \Omega$, the problem
is 
\begin{equation}\label{def:waveeqn}
  \begin{aligned}
  \partial_t^2 u - \Delta u &= 0 \quad\text{in } \Omega^c \times (0,T),\\ 
    u &= f \quad\text{on } \Gamma \times (0,T),\\
    \partial_t u = u &= 0 \quad \text{in } \Omega^c \times \{0\}.
  \end{aligned}
\end{equation}
The radiation condition is that $u$ has finite support at each time.
The solution is represented by the retarded single layer potential
\begin{equation}\label{def:single}
u(\xb,t) = \tilde \V q(\xb,t) := \int_{\Gamma} \frac{1}{4\pi\abs{\xb-\yb}}
q(\yb,t-\abs{\xb-\yb}) \, ds(\yb), \quad  \xb \in \R^3 \setminus \Gamma\,,
\end{equation}
where $\Gamma=\partial \Omega$ and $q$ is an unknown density on the
space time cylinder $\Gamma \times [0,T]$.  The jump relations of the
single layer potential lead to the boundary integral equation
\begin{equation}\label{def:intgrEqn}
  \V q(\xb,t) = f(\xb,t)  \qquad  (\xb,t) \in \Gamma.
\end{equation}
for $q$.
Here $\V$ is the operator obtained by letting $\xb$ in \eqref{def:single} be
on the boundary.

For the discretization we use a
tensor product of piecewise polynomial functions in space and
time. Specifically, for space, the surface $\Gamma \subset \R^3$ is subdivided
into conforming patches
\begin{equation}\label{def:triangul}
\Gamma = \bigcup_{j=1}^{N_f} P_j.
\end{equation}
We assume that the $P_j$'s are flat and ignore the additional
approximation error that occurs 
if \eqref{def:triangul} represents only an approximation of the actual
geometry of the problem. The space of piecewise polynomial
functions is spanned by functions
$\phi_i$, $i\in \{1,\dots,N_s\}$, where $N_s$ is the dimension of the
spatial ansatz space.

To simplify the exposition, we consider only piecewise constant or
linear elements on a uniform subdivision of the
interval $[0,T]$. If $N_t$ is the number of subintervals with length
$h=T/N_t$, then the
temporal ansatz functions are
\begin{equation*}
\varphi_k(t) = \varphi\left( \frac{t}{h} - k\right), \quad k\in\{0,\dots,N_t\},
\end{equation*}
where $\varphi$ is the usual box or hat function.

Thus the solution $q$ of \eqref{def:intgrEqn} is approximated by
\begin{equation}\label{def:coeff:q}
  q_h(\yb,\tau) = \sum_{\ell=0}^{N_t} \sum_{j=1}^{N_s}
  \varphi_\ell(\tau) \phi_j(\yb) q_{\ell j}.
\end{equation}
The Galerkin projection is obtained by multiplying with the time
derivative of the test functions. This leads to
\begin{equation}\label{varform}
  \int_0^\infty \int_{\Gamma} \varphi'_k(t) \phi_i(\xb) \V q_h(\xb,t) \, ds(\xb) dt
  = \int_0^\infty \int_{\Gamma} \varphi'_k(t) \phi_i(\xb) f(\xb,t) \, ds(\xb) dt
\end{equation}
for $k\in \{0,\dots,N_t\}$ and $i\in \{1,\dots,N_s\}$. The integrals on the left
hand side of this equation involve the time integrals
\begin{equation}\label{def:gd}
  \int_0^\infty \varphi'_{k}(t) \varphi_{\ell} (t-r) \, dt
  = \int_{-s}^\infty \varphi'(u-d) \varphi(u-r/h)\, du
  =: g_{k-\ell}(r)
\end{equation}
where $s=0$ when $\varphi$ is the box function and $s=1$ when
$\varphi$ is the hat function. Thus the variational form \eqref{varform}
reduces to the linear system 
\begin{equation}\label{def:linsys}
  \sum_{\ell=0}^k A_{k-\ell}\, q_\ell = f_k , \quad k=0,1,\dots N_t,
\end{equation}
where $q_\ell\in \R^{N_s}$ are the coefficients for the $\ell$-th time
step in \eqref{def:coeff:q}, $f_k$ is the vector with coefficients
\begin{equation*}
  f_{k,i} = \int_0^T \int_{\Gamma} \varphi'_k(t) \phi_i(\xb) f(\xb,t) \, ds(\xb) dt.
\end{equation*}
The matrix $A_{d} \in \R^{N_s\times N_s}$ is defined by
\begin{equation}\label{def:Aij}
  A_{d,ij} = \int_\Gamma \int_\Gamma \frac{1}{4\pi\abs{\xb-\yb}} 
  g_d(\abs{\xb-\yb}) \phi_i(\xb)\phi_j(\yb)\, ds(\yb) ds(\xb),
\end{equation}
where $-s \leq d \leq N_t$. Thus the system is block lower
triangular if piecewise constant ansatz functions are used whereas for
piecewise linears the system has an additional upper diagonal. 
It is straight forward to determine the explicit form of the time
integrated kernels. One finds that
\begin{equation*}
  g_d(r) =
  \begin{cases}
    0 & \text{ for } r < 0, \\ 
  \tilde g(r/h - d)  & \text{ for } r \geq 0,
\end{cases}
\end{equation*}
where in the case of piecewise constants
\begin{equation*}
\tilde g(z) =
\begin{cases}
   \tilde g_{-1} (z) := 1  & \text{ for } z\in [-1,0], \\
  \tilde g_{0} (z) :=  -1  & \text{ for } z\in [0,1],  \\  0 & \text{ otherwise.}
\end{cases}
\end{equation*}
and in the case of piecewise linears
\begin{equation*}
\tilde g(z) =
\begin{cases}
   \tilde g_{-2} (z) := \half \left(z+2\right)^2  & \text{ for } z\in [-2,-1],  \\
  \tilde g_{-1} (z) := -\half z \left(3z + 4\right) & \text{ for } z\in [-1,0],  \\
   \tilde g_{0} (z) := \half z \left(3z - 4\right) & \text{ for } z\in [0,1],  \\
  \tilde g_{1} (z) := -\half \left(z-2\right)^2  & \text{ for } z\in [1,2],  \\
  0 & \text{ otherwise.}
\end{cases}
\end{equation*}
The function $\tilde g(z)$ can also be written as
\begin{equation*}
\tilde g(z) = \sum_{j=-s-1}^s \tilde g_j(z) \chi_{[j,j+1]}(z)
\end{equation*}
where $\chi_{[j,j+1]}$ is the characteristic function of the interval
$[j,j+1]$. Using the identity $\chi_{[a,b]}(z) = H(b-z) - H(a-z)$ where $H$
is the Heavyside function, it follows that
\begin{equation*}
\tilde g(z) = \sum_{j=-s-1}^s \tilde g_j(z) \Big( H(j+1-z) - H(j-z) \Big)
= \sum_{j=-s-1}^{s+1} \tilde k_j(z) H(j-z)
\end{equation*}
where  $\tilde k_j = g_j - g_{j-1}$, $j\in \{-s,\dots,s\}$ and
$\tilde k_s = g_s$.
Returning to the function defined in \eqref{def:gd} shows that
\begin{equation*}
  g_d(z) = \sum_{j=-s-1}^{s+1} \tilde k_j\left(\frac{r}{h}-d\right)
                             H\left(j-d - \frac{r}{h}\right)
:= \sum_{j=-s-1}^{s+1} k_{d,j}(r) .
\end{equation*}
Here $k_{d,j}(r)$ is a polynomial when $r\in [0, \rho^*]$ and
vanishes for $r>\rho^*$. The truncation parameter is  $\rho^*=(d-j)h$.
Substituting this representation of $g_d$ into 
\eqref{def:Aij} shows that the matrix coefficients involve computing
a number of integrals of the type \eqref{def:integral}.

\section{Convex Polytopes}\label{sec:polytope}
The patches $P_k$ are a special case of two dimensional convex
polytopes in $\R^3$. This section reviews the basic concepts from
polytope theory that will be used in the ensuing discussion. For a
more complete introduction to this topic the reader is referred
to~\cite{grunbaum03,ziegler95}.

In general, a convex polytope in $\R^d$ is the convex hull of its
vertices
\begin{equation}\label{def:polytope}
  P = [\vb_0, \dots, \vb_n] :=
   \left\{ \sum_{i=0}^n \lambda_i \vb_i \colon \lambda_i \geq
   0, \sum_{i=0}^n \lambda_i = 1 \right\}.
\end{equation}
Any face of $P$ is the convex hull of a subset of the vertices. Faces are distinguished
based on their dimension. The 0-faces are the vertices, the 1-faces are the edges and
the faces with dimension one less than the polytope are facets.
The smallest affine plane that contains $P$ is denoted by $\aff(P)$. If
$\vb \in \aff(P)$ then the set
\begin{equation}
  H_P = \aff(P) - \vb
\end{equation}
is a linear space, which is independent of the choice of $\vb$. When
dealing with polytopes it is important to distinguish between the
space dimension $d$ and $\dim(P)$, which the dimension of $H_P$.

\subsection{Cartesian Products}
Since we will deal with Cartesian products of two 
polytopes frequently,  we recall some basic facts.
If $\{\vb_k\}_k$, $\{\wb_\ell\}_\ell$ are the vertices of two polytopes
$P_x$ and $P_y$ in $\R^d$, then $P = P_x\times P_y$ is a polytope in
$\R^{2d}$ whose vertices are
$\{(\vb_k,\wb_\ell)\}_{k,\ell}$. The dimension of $P$ is the product of
the dimensions of $P_x$ and $P_y$. 
The $j$-dimensional faces of
$P_x\times P_y$ are Cartesian products of the form $F_x\times F_y$,
where $F_x$ are the $j_x$-faces of $P_x$ , $F_y$ are the $j_y$-faces
of $P_y$ where $j_x + j_y = j$. This implies that the facets of
$P_x\times P_y$ are either products of $P_x$ with the facets of $P_y$
or products of the facets of $P_x$ with $P_y$ .

For a closed subset $F$ of $P_x\times P_y$ the minimal and maximal
distance are 
defined as 
\begin{equation*}
  r_{min}(F) = \min_{(\xb,\yb)\in F} \abs{\xb-\yb} \quad\mbox{and}\quad
  r_{max}(F) = \max_{(\xb,\yb)\in F} \abs{\xb-\yb}.
\end{equation*}

\subsection{Pyramids and Convex Hulls}
An important special case of convex polytopes are pyramids. These are
polytopes where all but one vertices are in an affine
plane. The vertex that is not in the plane is the apex of the
pyramid. The convex hull of the other vertices is a face of the
pyramid, called the base. The notation of a pyramid is
\begin{equation*}
  \pyr(\ab_0, B_0) := \left\{ (1-\lambda) \ab_0 + \lambda \xb_B :\;
    \xb_B\in B_0,\; \lambda \in [0,1] \right\}.
\end{equation*}
If the base is pyramid, i.e., $B_0=\pyr(\ab_1, B_1)$, then 
the notation $\pyr(\ab_0,\ab_1,B_1)$ is used to indicate such an iterated
pyramid. Likewise, $\pyr(\ab_0,\dots \ab_\ell,B_\ell)$ is an $\ell$-times
iterated pyramid.  If $B_0$ is the simplex with vertices
$\ab_0,\dots \ab_n$ then this process can be continued until the last
base is the vertex $\ab_n$.  Thus
$B_0 = \pyr(\ab_0,\dots,\ab_{n-1},[\ab_n])$.

We now turn to the convex hull of two polytopes 
$A$ and $B$ which is defined by
\begin{equation}\label{def:convAB}
  \conv(A,B) = \Big\{ (1-\lambda) \xb_A + \lambda \xb_B :
    \lambda \in [0,1],\; \xb_A \in A,\; \xb_B \in B \Big\} 
\end{equation}
While this definition applies to any two polytopes, we are only interested in situations
where $\xb_A$, $\xb_B$  and $\lambda$ are uniquely refined. This is
true when the following conditins are met
\begin{equation}\label{asu:convAB}
  \aff(A) \cap \aff(B) = \emptyset\quad\mbox{and}\quad
  H_A \cap H_B = \{ \mathbf{0} \}.
\end{equation}

The most important example of a convex hull that satisfies \eqref{asu:convAB}
is the iterated pyramid. If $A=[\ab_0,\dots\ab_k]$ is the simplex
spanned by the vertices then
\begin{equation*}
  \pyr(\ab_0,\dots, \ab_k, B) =
  \conv(A,B).
\end{equation*}
If in \eqref{def:convAB} the set $A$ is the convex hull of two sets
$S$ and $R$, then we obtain that $\conv(A,B) = \conv(S,R,B)$
where
\begin{equation*}
  \conv(S,R,B) = 
  \Big\{ (1-\lambda_1-\lambda_2) \zb_S + \lambda_1
  \zb_R + \lambda_2 \zb_B :\\
   \lambdab \in \Lambda,\, \zb_S \in S,\, \zb_R \in R,\, \zb_B \in B \Big\} 
\end{equation*}
is a convex hull of three sets. Here
$\Lambda = \{ \lambdab: 0\leq
\lambda_1,\lambda_2,\,\lambda_1+\lambda_2\leq 1\}$ is the standard
triangle.

\subsection{Integration over Polytopes}
The polytopes considered
in this article have a lower dimension than the space in which
they are located. Therefore integrals over polytopes are always understood to be in
the measure of their affine plane. Thus
if $F \subset \R^d$ is a $j$-dimensional polytope and the columns of
the matrix $Q \in \R^{d\times j}$ form an
orthogonal basis of $H_F$, then the integral of $f$ over $F$ is
\begin{equation*}
  \int_F f(\zb)\, d\zb = \int_{F_Q} f(\vb + Q\xib)\,d\xib,
\end{equation*}
where $\vb \in \aff(F)$ and the parameter space $F_Q$ is a polytope in $\R^j$ with nonempty
interior. If instead of an orthogonal basis we use a general basis of $H_F$ then it follows
\begin{equation*}
  \int_F f(\zb)\, d\zb = \det(R) \int_{F_T} f(\vb + T\xib)\,d\xib,
\end{equation*}
Here $T = QR$ is the qr-factorization, $R \in \R^{j\times j}$ is upper
triangular and $\det(R)$ is the
Jacobian of the transformation. If $\dim(F)=0$, i.e., the face is a
vertex $F=[\vb]$, then we will interpret the integral over $F$ in the point measure
\begin{equation*}
  \int_{[\vb]} f(\zb)\, d\zb = f(\vb)\,.
\end{equation*}
If $f$ is a smooth function, then standard quadrature methods can
be easily modified to compute the integral. For instance if $F_T$ is a
standard domain, such as a cube or a simplex, one could use tensor
product Gauss
quadrature or Gauss-like quadrature rules for simplices.  

Using the transformation $\zb = (1-\lambda)\zb_A + \lambda \zb_B$ an
integral over $\conv(A,B)$ can be reduced to an iterated integral
as follows
\begin{equation}\label{intgr:convAB}
  \int\limits_{\conv(A,B)} \!\!\!\! f(\zb)\, d\zb =
  \delta_{AB} \int\limits_{A}\int\limits_{B} \int\limits_0^1  
  g\big((1-\lambda) \zb_A + \lambda \zb_B\big)
  (1-\lambda)^{d_A} \lambda^{d_B} \,
    d\lambda \,d\zb_B d\zb_A.
\end{equation}
Here $d_A$ and $d_B$ are the dimensions of $A$ and $B$ and
$\delta_{AB}(1-\lambda)^{d_A} \lambda^{d_B}$ is the Jacobian of the
transformation. The
constant $\delta_{AB}$ expresses the geometric positioning of the
polytopes $A$ and $B$. To compute it, let $T_A$ and $T_B$ be bases of
$H_A$ and $H_B$. Further, let $\vb\in\aff(A)$ and $\wb\in\aff(B)$, then
$T = [T_A, T_b, \vb-\wb]$ is a basis of $H_{\conv(A,B)}$. Let
$R_A$, $R_B$ and $R$ be the triangular factors in the
qr-factorizations of $T$, $T_A$ and $T_B$, then 
\begin{equation*}
 \delta_{AB} = \frac{ \det(R)}{\det(R_A)\det(R_B)}.
\end{equation*}
Using elementary linear algebra one can verify that $\delta_{AB}$ is
independent of the choice of the bases and independent of the choice
of vectors $\vb$ and $\wb$.

The form of \eqref{intgr:convAB} is particularly useful if the
integrand has a singularity of the type
\begin{equation*}
  f( (1-\lambda)\zb_A + \lambda \zb_B ) = \lambda^{-\alpha}
  g( (1-\lambda)\zb_A + \lambda \zb_B )
\end{equation*}
where $g$ is smooth. In that case, the $\lambda$-integral simplifies
to
\begin{equation*}
  I(\zb_A, \zb_B) = \int\limits_0^1
  g( (1-\lambda)\zb_A + \lambda \zb_B ) (1-\lambda)^{d_A}
  \lambda^{d_B-\alpha} \, d\lambda
\end{equation*}
Provided that $\alpha < d_B+1$ this integral can be effectively
computed with a Gauss-Jacobi rule for the weight function
$(1-\lambda)^{d_A}\lambda^{d_B-\alpha}$. In addition, $I(\zb_A,
\zb_B)$ is a smooth, hence the integrals over $A$ and $B$ can
be effectively approximated by regular quadrature rules.

In the spirit of \eqref{intgr:convAB} one can see that an
integral over the triple convex hull becomes the iterated integral
\begin{multline}\label{intgr:convSRB}
  \int\limits_{\conv(S,R,B)} \!\!\!\! f(\zb)\, d\zb = \delta_{SRB}
  \times \\
  \int\limits_{S}\int\limits_{R} \int\limits_{B}\int\limits_{\Lambda}
  f\big((1-\lambda_1-\lambda_2) \zb_S + \lambda_1 \zb_R \lambda_2 + \zb_B\big)
  (1-\lambda_1-\lambda_2)^{d_S} \lambda_1^{d_R}\lambda_2^{d_B} \,
    d\lambdab d\zb_B d\zb_R \zb_S.
\end{multline}
where $\delta_{SRB}$ is defined in terms of 
qr-factorizations. Specifically, let $T_S$, $T_R$ and $T_B$ be
bases of the linear spaces of $S$, $R$ and $B$, $\ub\in \aff(S)$,
$\vb\in \aff(R)$ and $\wb\in \aff(B)$, 
and set $T=[T_S, T_R, T_B, \ub-\vb, \ub-\wb]$ , then
\begin{equation*}
 \delta_{SRB} = \frac{ \det(R)}{\det(R_S)\det(R_R)\det(R_B)},
\end{equation*}
where $R$, $R_S$ $R_R$ are $R_B$ are the upper triangular matrices in
the qr-factorization of $T$, $T_S$, $T_R$ and $T_B$. This form enables
the treatment of certain singularities of the integrand which will be
used in Section~\ref{sec:analysis}.

\section{Pyramidal decomposition}\label{sec:decomp}
This section discusses how a general polytope $P$ can be decomposed into iterated pyramids.
To that end, select any point $\ab\in P$ as the apex and denote by $\faces{\ab}{P}$ the
set of facets of $P$ that do not contain the point $\ab$.  Then $P$ can
be decomposed into pyramids with the apex $\ab$ and bases in $\faces{\ab}{P}$.
That is,
\begin{equation}\label{decomp:P}
  P = \bigcup\limits_{F \in \faces{\ab}{P}} \!\! \pyr(\ab, F),
\end{equation}
where the union has disjoint interiors. 
This simple fact is illustrated in Figure~\ref{fig:cubesDecomp}.
The proof of \eqref{decomp:P}
for the case that the apex is a vertex is in \cite{tausch26}. Since
this proof easily extends to the case that $\ab$ is an arbitrary point
of $P$ it is omitted here.

A few comments on the set $\faces{\ab}{P}$ are in order. If the apex
is in the relative interior of $P$ then $\faces{\ab}{P}$ is the set of
all facets of $P$. If $\ab$ is on a lower dimensional face of $P$ then
all facets are excluded that contain the apex. Hence the cardinality
of $\faces{\ab}{P}$ is minimal for a certain vertex of $P$, and
selecting vertices for the apex generally lead to smaller
decompositions. 
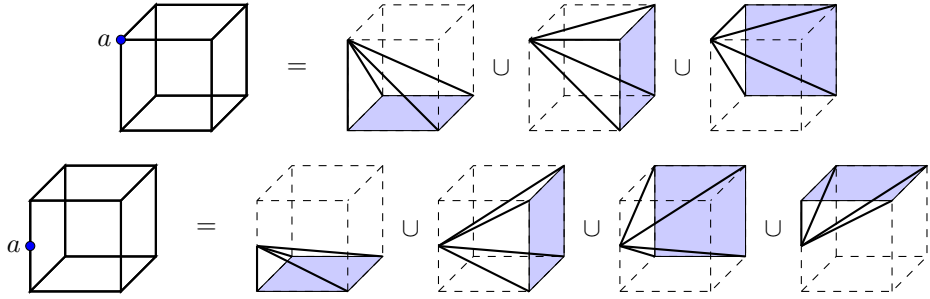
\begin{figure}
\begin{center}
  %
%
\begin{tikzpicture}[xscale=0.6, yscale=0.6] 
  \coordinate (A) at (0,0,0);
  \coordinate (B) at (2,0,0);
  \coordinate (C) at (0,2,0);
  \coordinate (D) at (2,2,0);
  \coordinate (E) at (0,0,2);
  \coordinate (F) at (2,0,2);
  \coordinate (G) at (0,2,2);
  \coordinate (H) at (2,2,2);
\draw[thick] (A) -- (E) -- (F) -- (B) -- cycle;
\draw[thick] (A) -- (C) -- (G) -- (E) -- cycle;
\draw[thick] (A) -- (B) -- (D) -- (C) -- cycle;
\draw[thick] (E) -- (F) -- (H) -- (G) -- cycle;
\draw[thick] (F) -- (B) -- (D) -- (H) -- cycle;
\draw[thick] (G) -- (H) -- (D) -- (C) -- cycle;
\draw[fill=blue] (G) circle (0.10cm);
\node[left] at (G) {$a$};
\node[] at (3.5,1,1) {$=$};
%
%
  \coordinate (A) at (5,0,0);
  \coordinate (B) at (7,0,0);
  \coordinate (C) at (5,2,0);
  \coordinate (D) at (7,2,0);
  \coordinate (E) at (5,0,2);
  \coordinate (F) at (7,0,2);
  \coordinate (G) at (5,2,2);
  \coordinate (H) at (7,2,2);
\draw[fill=blue!20] (A) -- (E) -- (F) -- (B) -- cycle;
\draw[thick] (G) -- (A);
\draw[thick] (G) -- (E);
\draw[thick] (G) -- (F);
\draw[thick] (G) -- (B);
\draw[dashed] (E) -- (A) -- (B);
\draw[dashed] (A) -- (C);
\draw[dashed] (E) -- (F) -- (B) -- (D) -- (C) -- (G) -- cycle;
\draw[dashed] (F) -- (H) -- (G);
\draw[dashed] (D) -- (H);
\node[] at (8,1,1) {$\cup$};
%
%
  \coordinate (A) at (9,0,0);
  \coordinate (B) at (11,0,0);
  \coordinate (C) at (9,2,0);
  \coordinate (D) at (11,2,0);
  \coordinate (E) at (9,0,2);
  \coordinate (F) at (11,0,2);
  \coordinate (G) at (9,2,2);
  \coordinate (H) at (11,2,2);
\draw[fill=blue!20] (F) -- (B) -- (D) -- (H) -- cycle;
\draw[thick] (G) -- (F);
\draw[thick] (G) -- (B);
\draw[thick] (G) -- (D);
\draw[thick] (G) -- (H);
\draw[dashed] (E) -- (A) -- (B);
\draw[dashed] (A) -- (C);
\draw[dashed] (E) -- (F) -- (B) -- (D) -- (C) -- (G) -- cycle;
\draw[dashed] (F) -- (H) -- (G);
\draw[dashed] (D) -- (H);
\node[] at (12,1,1) {$\cup$};
%
%
  \coordinate (A) at (13,0,0);
  \coordinate (B) at (15,0,0);
  \coordinate (C) at (13,2,0);
  \coordinate (D) at (15,2,0);
  \coordinate (E) at (13,0,2);
  \coordinate (F) at (15,0,2);
  \coordinate (G) at (13,2,2);
  \coordinate (H) at (15,2,2);
\draw[fill=blue!20] (A) -- (B) -- (D) -- (C) -- cycle;
\draw[thick] (G) -- (A);
\draw[thick] (G) -- (B);
\draw[thick] (G) -- (D);
\draw[thick] (G) -- (C);
\draw[dashed] (E) -- (A) -- (B);
\draw[dashed] (A) -- (C);
\draw[dashed] (E) -- (F) -- (B) -- (D) -- (C) -- (G) -- cycle;
\draw[dashed] (F) -- (H) -- (G);
\draw[dashed] (D) -- (H);
\end{tikzpicture}

\vspace*{0.4cm}
%
%
\begin{tikzpicture}[xscale=0.6, yscale=0.6] 
  \coordinate (A) at (0,0,0);
  \coordinate (B) at (2,0,0);
  \coordinate (C) at (0,2,0);
  \coordinate (D) at (2,2,0);
  \coordinate (E) at (0,0,2);
  \coordinate (F) at (2,0,2);
  \coordinate (G) at (0,2,2);
  \coordinate (H) at (2,2,2);
  \coordinate (K) at (0,1,2);
\draw[thick] (A) -- (E) -- (F) -- (B) -- cycle;
\draw[thick] (A) -- (C) -- (G) -- (E) -- cycle;
\draw[thick] (A) -- (B) -- (D) -- (C) -- cycle;
\draw[thick] (E) -- (F) -- (H) -- (G) -- cycle;
\draw[thick] (F) -- (B) -- (D) -- (H) -- cycle;
\draw[thick] (G) -- (H) -- (D) -- (C) -- cycle;
\draw[fill=blue] (K) circle (0.10cm);
\node[left] at (K) {$a$};
\node[] at (3.5,1,1) {$=$};
%
%
  \coordinate (A) at (5,0,0);
  \coordinate (B) at (7,0,0);
  \coordinate (C) at (5,2,0);
  \coordinate (D) at (7,2,0);
  \coordinate (E) at (5,0,2);
  \coordinate (F) at (7,0,2);
  \coordinate (G) at (5,2,2);
  \coordinate (H) at (7,2,2);
  \coordinate (K) at (5,1,2);
\draw[fill=blue!20] (A) -- (E) -- (F) -- (B) -- cycle;
\draw[thick] (K) -- (A);
\draw[thick] (K) -- (E);
\draw[thick] (K) -- (F);
\draw[thick] (K) -- (B);
\draw[dashed] (E) -- (A) -- (B);
\draw[dashed] (A) -- (C);
\draw[dashed] (E) -- (F) -- (B) -- (D) -- (C) -- (G) -- cycle;
\draw[dashed] (F) -- (H) -- (G);
\draw[dashed] (D) -- (H);
\node[] at (8,1,1) {$\cup$};
%
%
  \coordinate (A) at (9,0,0);
  \coordinate (B) at (11,0,0);
  \coordinate (C) at (9,2,0);
  \coordinate (D) at (11,2,0);
  \coordinate (E) at (9,0,2);
  \coordinate (F) at (11,0,2);
  \coordinate (G) at (9,2,2);
  \coordinate (H) at (11,2,2);
  \coordinate (K) at (9,1,2);
\draw[fill=blue!20] (F) -- (B) -- (D) -- (H) -- cycle;
\draw[thick] (K) -- (F);
\draw[thick] (K) -- (B);
\draw[thick] (K) -- (D);
\draw[thick] (K) -- (H);
\draw[dashed] (E) -- (A) -- (B);
\draw[dashed] (A) -- (C);
\draw[dashed] (E) -- (F) -- (B) -- (D) -- (C) -- (G) -- cycle;
\draw[dashed] (F) -- (H) -- (G);
\draw[dashed] (D) -- (H);
\node[] at (12,1,1) {$\cup$};
%
%
  \coordinate (A) at (13,0,0);
  \coordinate (B) at (15,0,0);
  \coordinate (C) at (13,2,0);
  \coordinate (D) at (15,2,0);
  \coordinate (E) at (13,0,2);
  \coordinate (F) at (15,0,2);
  \coordinate (G) at (13,2,2);
  \coordinate (H) at (15,2,2);
  \coordinate (K) at (13,1,2);
\draw[fill=blue!20] (A) -- (B) -- (D) -- (C) -- cycle;
\draw[thick] (K) -- (A);
\draw[thick] (K) -- (B);
\draw[thick] (K) -- (D);
\draw[thick] (K) -- (C);
\draw[dashed] (E) -- (A) -- (B);
\draw[dashed] (A) -- (C);
\draw[dashed] (E) -- (F) -- (B) -- (D) -- (C) -- (G) -- cycle;
\draw[dashed] (F) -- (H) -- (G);
\draw[dashed] (D) -- (H);
\node[] at (16,1,1) {$\cup$};
%
%
  \coordinate (A) at (17,0,0);
  \coordinate (B) at (19,0,0);
  \coordinate (C) at (17,2,0);
  \coordinate (D) at (19,2,0);
  \coordinate (E) at (17,0,2);
  \coordinate (F) at (19,0,2);
  \coordinate (G) at (17,2,2);
  \coordinate (H) at (19,2,2);
  \coordinate (K) at (17,1,2);
\draw[fill=blue!20] (G) -- (H) -- (D) -- (C) -- cycle;
\draw[thick] (K) -- (G);
\draw[thick] (K) -- (D);
\draw[thick] (K) -- (C);
\draw[thick] (K) -- (H);
\draw[dashed] (E) -- (A) -- (B);
\draw[dashed] (A) -- (C);
\draw[dashed] (E) -- (F) -- (B) -- (D) -- (C) -- (G) -- cycle;
\draw[dashed] (F) -- (H) -- (G);
\draw[dashed] (D) -- (H);
\end{tikzpicture}  
\end{center}
\caption{Decomposition of a cube into pyramids. If the apex is a
  vertex, then there are three faces that do not contain the apex. If
  the apex is on an edge the number of faces increases to four. Not
  shown: An apex on a face results in five and an apex in the interior
results in six pyramids.} 
\label{fig:cubesDecomp}
\end{figure}

\subsection{Pyramidal Decomposition Algorithm} 
The pyramidal decomposition can be repeated for any or all facets in
$\faces{\ab}{P}$. This leads to a decomposition of $P$ into iterated
pyramids with disjoint interiors. This construction will be referred
to as the pyramidal decomposition of $P$. In the extreme case, the
iteration can be continued until the last bases are vertices. In this
case the pyramidal decomposition leads to a decomposition of $P$ into
simplices. However, as we will see below, the decomposition is usually
stopped sooner, until all final bases satisfy a certain specified
stopping criterion. In addition, the final decomposition will depend
on how the apices are selected.
We will denote by $\leaves{P}$ the set of final bases and
write $K=\#\leaves{P}$.
The recursive algorithm is described in~\ref{algo:pyraDecomp}.
\begin{algorithm}\label{alg:pyraDecomp}
\caption{Pyramidal Decomposition Algorithm}
\label{algo:pyraDecomp}
\begin{algorithmic}
\Function {pyraDecomp}{$F$}
\If{$F$ satisfies the stopping criterium}
\State \Return
\Else
  \State Select an apex $\ab\in F$ 
  \ForAll{$F' \in \faces{\ab}{F}$}
     \State \Call{pyraDecomp}{$F'$}
  \EndFor 
\EndIf
\EndFunction
\end{algorithmic}
\end{algorithm}

The decomposition algorithm can be applied to any type of polytope. In
particular, in the following we will work with 
decompositions of Cartesian
product $P=P_x\times P_y$. For such a polytope, the apex is 
$\ab=(\ab_x,\ab_y)$ where $\ab_x \in P_x$ and $\ab_y \in P_y$
and facet sets can be determined from the facets of $P_x$ and  $P_y$
as follows
\begin{equation*}
  \faces{\ab}{P} =
  \Big\{ F_x \times P_y : F_x \in \faces{\ab_x}{P_x} \Big\} \cup
  \Big\{ P_x \times F_y : F_x \in \faces{\ab_y}{P_y} \Big\}.
\end{equation*}
By recursion, we obtain a decomposition into convex hulls
\begin{equation}\label{decompPxy}
P_x\times P_y = \bigcup_{k=1}^K \conv(A_k,B_k) 
\end{equation}
where 
\begin{equation*}
  \begin{aligned}
    A_k &= \big[ \ab_0^{(k)} \dots, \ab^{(k)}_{\ell_k}\big]\\ 
    B_k &= F_x^{(k)} \times F_y^{(k)}
  \end{aligned}
\end{equation*}
Here, $A_k$ is an $\ell_k$-dimensional simplex in $\R^6$ and $F_x^{(k)}$ and
$F_y^{(k)}$ are faces of $P_x$ and $P_y$. The number $\ell_k$ is
bounded by the dimension of $P_x\times P_y$.

\subsection{Decomposition Algorithm for Truncated Kernels}
The final decomposition of algorithm~\ref{alg:pyraDecomp} depends on
the stopping criterion and the way the apices are selected. This
section describes a particular choice that is appropriate for the
integrals \eqref{def:integral} where the truncation parameter of the
kernel is $\rho^*$.

First, we introduce some terminology. For faces $F_x\subset P_x$ and
$F_y\subset P_y$ we say that $F=F_x\times F_y$ is
singular if $F_x\cap F_y$ is nonempty. This is equivalent with
$r_{min}(F)=0$.  Since subdivision \eqref{def:triangul} is assumed to be conforming,
the intersection is a common face. If the intersection is empty we say
that $F$ is nonsingular.  If $r_{min}(F)\geq\rho^*$ we say that $F$ is
far, if $r_{max}(F)\leq\rho^*$ then $F$ is near.  If $F$ is neither
near nor far it is mixed, which is equivalent to
$r_{min}(F)< \rho^*< r_{max}(F)$.

\begin{definition}\label{def:stop}
  \emph{Stopping Criterium.} The pyramidal decomposition is stopped if
  $F$ is nonsingular and either near or far.
\end{definition}

\begin{definition}\label{def:apex}
\emph{Apex Selection.} If $F=F_x\times F_y$ is singular 
the apex is $\ab=(\ab_x,\ab_x)$ where $\ab_x$ is a common vertex in
$F_x$ and $F_y$. If $F$ is nonsingular, then choose $\ab_x \in F_x$
and $\ab_y \in F_y$ such that $\abs{\ab_x - \ab_y} < \rho^*$.
\end{definition}

\paragraph{Remarks}
\begin{enumerate}
  \item An apex has to be selected only for faces that do not satisfy the
    stopping criterion. This is the case when $F$ is singular or
    otherwise, if $F$ is mixed. In the former case there is always a vertex in
    the intersection. In the latter case there is always
    a pair of points in $\ab_x \in F_x$ and $\ab_y\in F_y$ such that
    $\abs{\ab_x-\ab_y}<\rho^*$. 
  \item Unless there is exactly one common vertex the choice of apex
    is not unique. The selection of the apex has an effect on final the number of
    convex hulls $K$. Here the general strategy is to minimize the set
    $\faces{\ab}{F}$. This will be discussed in Section~\ref{sec:implementation}.
\end{enumerate}  
Figure~\ref{fig:pyraLattice} illustrates the resulting decomposition for a pair of mixed triangles.

\begin{figure}
\begin{center}
 \begin{tikzpicture}[xscale=1.65, yscale=1.6]
  \node (A) at (5,2){\scalebox{0.8}{$[\vb_0 \vb_1 \vb_2] \times [\wb_0 \wb_1 \wb_2]$}};
  \node (B) at (4,1){\scalebox{0.8}{$[\vb_1 \vb_2] \times [\wb_0 \wb_1 \wb_2]$}};
  \node (C) at (6.5,1){\scalebox{0.8}{$[\vb_1 \vb_0 \vb_2] \times [\wb_1 \wb_2]$}};
  \node (D) at (3,0){\scalebox{0.8}{$[\vb_2] \times [\wb_0 \wb_1 \wb_2]$}};
  \node (E) at (5,0){\scalebox{0.8}{$[\vb_1 \vb_2] \times [\wb_1 \wb_2]$}};
  \draw[-,thick] (A)--(B);
  \draw[-,thick] (A)--(C);
  \draw[-,thick] (B)--(D);
  \draw[-,thick] (B)--(E);
  \coordinate (v1) at (0.7,2);
  \coordinate (v2) at (0,1.5);
  \coordinate (v0) at (0.8,0.8);
  \coordinate (w0) at (1.2,1.3);
  \coordinate (w2) at (2,2);
  \coordinate (w1) at (1.9,0.8);
  \draw[-,thick] (v0)--(v1);
  \draw[-,thick] (v1)--(v2);
  \draw[-,thick] (v2)--(v0);
  \draw[-,thick] (w0)--(w1);
  \draw[-,thick] (w1)--(w2);
  \draw[-,thick] (w2)--(w0);
  \draw[dashed] (v0)--(w0);
  \draw[dashed] (v1)--(w0);
  \node[below] at (v0) {\scalebox{0.8}{$\vb_0$}};
  \node[below] at (v2) {\scalebox{0.8}{$\vb_2$}};
  \node[above] at (v1) {\scalebox{0.8}{$\vb_1$}};
  \node[left] at (w0) {\scalebox{0.8}{$\wb_0$}};
  \node[right] at (w1) {\scalebox{0.8}{$\wb_1$}};
  \node[right] at (w2) {\scalebox{0.8}{$\wb_2$}};
\end{tikzpicture}
\end{center}
\caption{Pyramidal decomposition of two nonintersecting triangles.
  Dashed lines indicate pairs of vertices with distances less than
  $R$. In the singular faces the chosen apex is the Cartesian product
  of the first vertices. In this example the set $\leaves{P}$ consists
  of three faces that satisfy $r_{min}<\rho^*$.} 
\label{fig:pyraLattice}
\end{figure}
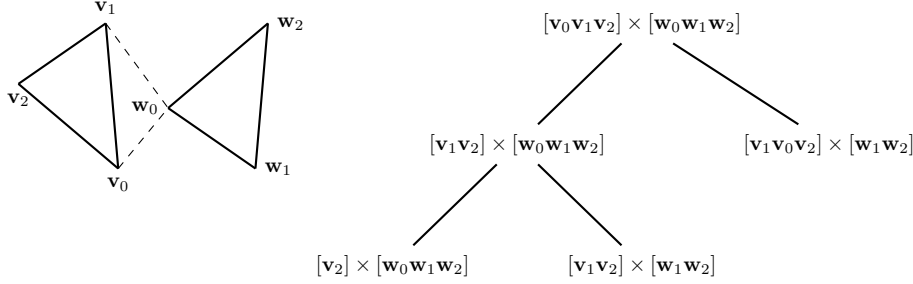

\section{Properties of the Pyramidal Decomposition}\label{sec:analysis}
This section collects some simple properties of the final iterated
pyramids in Algorithm~\ref{alg:pyraDecomp} and shows how the
singularities and discontinuities of the
resulting convex hull integrals can be treated.

For $B\in \leaves{P_x\times P_y}$ let $\ab_0,\dots,\ab_\ell$ be
the sequence of apices that lead to $B$. Further, $F_0, \dots
F_{\ell}$ denotes the sequence of faces that came before the
final base $B$. Thus $F_0 = P_x\times P_y$, $\ab_j \in F_j$,
$F_{j+1} \in \faces{\ab_j}{F_j}$ and $B_\ell \in \faces{\ab_\ell}{F_\ell}$.
Further,
$A=[\ab_0,\dots,\ab_\ell]$ is the simplex in $\R^{2d}$ that is generated by the
apices.

\begin{lemma}\label{lem:Anear}
It holds that
\begin{equation}
r_{max}(A) < \rho^* \quad\mbox{and}\quad r_{min}(B) > 0.
\end{equation}
\end{lemma}

\begin{proof}
To show the first inequality write $\ab_k=(\ab_{x,k},\ab_{y,k})$ for the $x$- and $y$-components of
the apices. By construction of the algorithm $\abs{\ab_{x,k}-\ab_{y,k})}<\rho^*$.
To show that all points $(\xb,\yb)$ in the simplex also satisfy
$\abs{\xb-\yb} < \rho^*$, note that the function
$\rho^2(\xb,\yb) = \abs{\xb-\yb}^2$ is convex, because for two points
$(\xb_0,\yb_0)$, $(\xb_1,\yb_1)$ in $A$ the function
\begin{equation*}
  \lambda \mapsto \abs{(1-\lambda)(\xb_0-\yb_0) + \lambda(\xb_1-\yb_1) }^2
\end{equation*}
is either a constant or a parabola opening upwards. By convexity,
the maximum is on the boundary of the domain, which in
this case is a facet of $A$. Since $\rho^2$ is also convex on the
facet one can repeat the argument until it follows that the maximum
occurs on the vertices of $A$. This proves the first assertion. The
second assertion follows readily from the stopping criterion.
\end{proof}  

To derive the appropriate integration approach for the convex hulls
generated by Algorithm~\ref{alg:pyraDecomp}
we will distinguish whether the basis $B$ is near or far and whether all apices
in $A$ are singular, regular or contain types. Thus there is a total
of six different cases.

\subsection{All apices of $A$ are singular}
In this case $r_{max}(A) = 0$. For a point $\zb=(\xb,\yb) \in \conv(A,B)$ we have unique $\lambda\in [0,1]$,
$\zb_A=(\xb_A,\xb_A) \in A$ and $\zb_B=(\xb_B,\yb_B) \in B$ such that
$\zb = (1-\lambda) \zb_A + \lambda \zb_B$.
The distance function simplifies to
\begin{equation}\label{def:rho}
  \rho(\lambda) := \abs{\xb-\yb} =
 \lambda \abs{ (\xb_B-\yb_B)}.
\end{equation}
If $B$ is far, the intersection of $\rho(\lambda) = \rho^*$ is 
\begin{equation}\label{def:lambdastar}
\lambda^*= \frac{\rho^*}{\abs{\xb_B-\yb_B}},
\end{equation}
which is in the interval $(0,1]$ and a smooth function of $\zb_B$. If
$B$ is near then $0<\rho(\lambda) \leq \rho^*$ in $(0,1]$.

We turn to the integral \eqref{def:integral} the domain of
integration is replaced by the subset $\conv(A,B)$. The integrand 
can be written as
\begin{equation}\label{def:integrand1}
  \frac{k\big(\rho(\lambda)\big)}{\rho(\lambda)}
  \Phi\big( (1-\lambda) \zb_A + \lambda \zb_B\big)
 = \frac{g(\lambda,\zb_A,\zb_B)}{\lambda}
\end{equation}
where $g:[0,1]\times A\times B \mapsto \R$ is a smooth function. Since
the domain is a convex hull we use 
\eqref{intgr:convAB}. Thus the inner integral becomes
\begin{equation*}
I(\zb_A,\zb_B) = \int_0^{\lambda^*} g(\lambda,\zb_A,\zb_B)
(1-\lambda)^{d_A}\lambda^{d_B-1} \,d\lambda.
\end{equation*}
Here the denominator in \eqref{def:lambdastar} is absorbed in the
Jacobian. The smallest dimension of $B$ occurs when $P_x=P_y$ in which
case $d_B=1$. Thus the integrand is smooth. Moreover, the upper bound
is either \eqref{def:lambdastar} which is a smooth function or
$\lambda=1$. This implies that in both cases, when $B$ is near or far,
the function $I : A\times B \mapsto \R$ is smooth.

\subsection{All apices of $A$ are nonsingular}
In this case $r_{min}(A) > 0$.
For $\zb = (1-\lambda) \zb_A + \lambda \zb_B \in \conv(A,B)$
the distance function becomes
\begin{equation}\label{def:rho2}
  \rho(\lambda) = \abs{\xb-\yb} =
  \abs{ (1-\lambda)(\xb_A-\yb_A) + \lambda(\xb_B-\yb_B)},
\end{equation}
where $\rho(0) = \abs{\xb_A-\yb_A}$ and $\rho(1) = \abs{\xb_B-\yb_B}$.
Note that $\rho^2(\lambda)$ is either a parabola that opens up, or a
constant when $\rho(0)=\rho(1)$.  When $B$ is far we have
$\rho(0)<\rho^*$ and $\rho(1)\geq\rho^*$ and thus there is
$\lambda^* \in (0,1]$ such that $\rho(\lambda^*) = \rho^*$. The point
is unique, because otherwise the parabola would have a maximum between
the two solutions which is impossible. Further,
$\lambda^*=\lambda^*(\zb_A,\zb_B)$ is a smooth function because
$\rho(\lambda^*)>0$.  When $B$ is near we have $\rho(0)<\rho^*$ and
$\rho(1)\leq\rho^*$ and it follows from a similar maximum argument
that $0<\rho(\lambda)\leq\rho^*$ for $\lambda \in [0,1]$.

We turn to the inner integral that results when the domain in
\eqref{def:integral} is replaced by $\conv(A,B)$. Since $\rho(\lambda)$ is
positive, the integrand 
\begin{equation}\label{def:integrand2}
\frac{k\big(\rho(\lambda)\big)}{\rho(\lambda)}
  \Phi\big( (1-\lambda) \zb_A + \lambda \zb_B\big)
 =: g(\lambda,\zb_A,\zb_B)
\end{equation}
is a smooth function in $[0,1]\times A\times B$. Thus the inner
integral of \eqref{intgr:convAB} becomes
\begin{equation*}
I(\zb_A,\zb_B) = \int_0^{\lambda^*} g(\lambda,\zb_A,\zb_B)
(1-\lambda)^{d_A}\lambda^{d_B} \,d\lambda
\end{equation*}
which is a smooth function whenever $B$ is near or far.

\subsection{$A$ has both singular and nonsingular apices}
In this case $A$ is the convex hull of the simplices
$S=[\ab_0,\dots\ab_{s}]$ and $R=[\ab_{s+1},\dots, \ab_{\ell}]$, where
$0\leq s < \ell$. Thus $\conv(A,B) = \conv(S,R,B)$ is a triple convex
hull. Then
\begin{equation*}
\zb = (1-\lambda_1 - \lambda_2) \zb_S + \lambda_1 \zb_R + \lambda_2 \zb_B
\end{equation*}
where $\lambdab=(\lambda_1,\lambda_2)$ is in the standard triangle $\Lambda$.
Since $\xb_S=\yb_S$ the corresponding distance function is
\begin{equation}\label{def:rho3}
\rho(\lambdab) = \abs{\xb-\yb}
= \abs{ \lambda_1 (\xb_R-\yb_R) + \lambda_2 (\xb_B-\yb_B) }.
\end{equation}
The graph of $\lambdab \mapsto \rho^2(\lambdab)$ is either a plane or an
elliptic paraboloid that opens up. 

\begin{lemma}\label{lem:rho3}
It holds that
\begin{equation*}
    \rho(\lambdab) > 0 \;\mbox{for}\; \lambdab \in \Lambda \setminus \mathbf{0}.
\end{equation*}
\end{lemma}

\begin{proof}
  Consider the equation
\begin{equation}\label{zeroInConv}
\lambda_1 (\xb_R - \yb_R) + \lambda_2 (\xb_R - \yb_R) = \mathbf{0}.
\end{equation}
If $(\lambda_1,\lambda_2) \in \Lambda \setminus \mathbf{0}$ was a solution then we can
divide the equation by $\lambda_1+\lambda_2$ to get
\begin{equation*}
\mu_1 (\xb_R - \yb_R) + \mu_2 (\xb_R - \yb_R) = \mathbf{0}
\end{equation*}
with $\mu_i = \lambda_i/(\lambda_1+\lambda_2)$, $i=1,2$. Since
$\mu_1+\mu_2=1$ it then follows that $\mathbf{0} \in
\conv(R,B)$. By the design of the algorithm 
$\conv(R,B) \subset F_{s+1}$, where $\ab_s$ is the last singular
vertex. Since $F_{s+1}$ is therefore nonsingular we arrive at a
contradiction. Thus $\lambdab=0$ is the only solution to \eqref{zeroInConv}.
\end{proof}

We have to reduce the $\lambdab$-integral in
\eqref{intgr:convSRB} to the intersection of $\Lambda$ with the
interior of the ellipse
\begin{equation}\label{def:LambdaStar}
\Lambda^* = \{ \lambdab \in \Lambda : \rho(\lambdab) \leq R \}.
\end{equation}

Since $r_{max}(S)=0$ and $r_{max}(R)<\rho^*$, it follows that $\rho(0,0)=0$,
$\rho(0,1)<\rho^*$. If $B$ is near, i.e., $r_{max}(B) < \rho^*$ then
$\rho(1,0)<R$ and therefore $\rho<\rho^*$ on the three vertices of
$\Lambda$. From convexity it follows that $\Lambda$ is entirely inside
the ellipse and thus $\Lambda^* = \Lambda$.
If $B$ is far, i.e., $r_{min}(B) \geq \rho^*$ then $\rho(1,0)\geq \rho^*$. In
this case the ellipse has one intersection point with each edge
$[(0,0), (0,1)]$ and $[(1,0), (0,1)]$.

We now turn to the inner integral~\eqref{intgr:convSRB}. Setting
\begin{equation*}
  g(\lambda,\zb_S,\zb_R,\zb_B) =
k\big(\rho(\lambdab)\big)
  \Phi\big( (1-\lambda_1-\lambda_2) \zb_S + \lambda_1 \zb_R + \lambda_2 \zb_B\big)
\end{equation*}
we have to show that the function
\begin{equation*}
  I(\zb_S,\zb_R,\zb_B) := \int_{\Lambda^*} 
  \frac{g(\lambdab,\zb_S,\zb_R,\zb_B)}{\rho(\lambdab)}
  (1-\lambda_1 -\lambda_2)^{d_S}\lambda_1^{d_R} \lambda_2^{d_B} \,d\lambdab
\end{equation*}
is smooth.  By Lemma~\ref{lem:rho3} the only critical point is the origin,
where the denominator and the Jacobian vanish.  However,
singularities will appear in the higher derivatives of the
fraction. To remedy this introduce the change of variables
\begin{equation*}
  \begin{aligned}
    \lambda_1 &= \xi (1-\eta),\\
    \lambda_2 &= \xi \eta.
  \end{aligned}
\end{equation*}
The Jacobian is $\xi$ and 
the distance function becomes
\begin{equation}\label{rhoInXiEta}
\rho\big( (1-\eta)\xi,\eta\xi \big) = \xi \rho\big( (1-\eta),\eta \big),
\end{equation}
Thus the integral is transformed to
\begin{equation}\label{def:I:lambda3}
  I(\zb_S,\zb_R,\zb_B) = \int_{X^*} 
  \frac{g(\xi,\eta,\zb_S,\zb_R,\zb_B)}{\rho\big( (1-\eta),\eta \big)}
  \xi^{d_R+d_B} (1-\xi)^{d_S}(1-\eta)^{d_R} \eta^{d_B} \,d\xi d\eta.
\end{equation}
In the denominator the factor $\xi$ from \eqref{rhoInXiEta}
cancels with the Jacobian.
The remaining denominator is bounded away from zero and the integrand
is now a smooth function. 

The domain $X^*$ is the parameter set of the transformation. It is
shown in Figure~\ref{fig:trnsDmn}.
If $B$ is near, then $X^* = [0,1]^2$ is the unit square . If
$B$ is far, the domain decomposes into two pieces
$X^* = X^*_1 \cup X^*_2 $, where \eqref{def:LambdaStar}
\begin{equation*}
  \begin{aligned}
    X^*_1 &= [0,1]\times [0,\eta^*],\\
    X^*_2 &= \big\{ (\xi,\eta) :  0 \leq \xi \leq \xi^*(\eta),\;
            \eta* \leq \eta \leq 1 \big \}.
  \end{aligned}
\end{equation*}

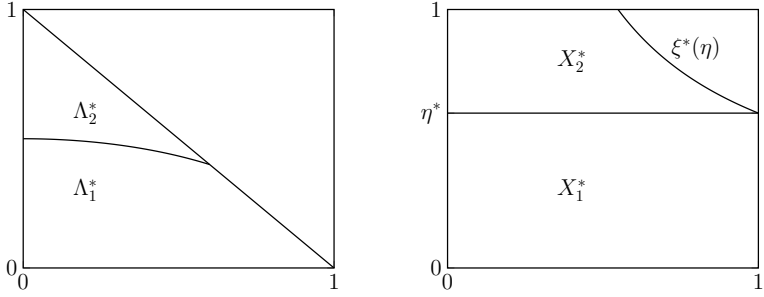
\begin{figure}
\begin{center}
\centerline{
\begin{tikzpicture}[xscale=0.6, yscale=0.6]
\begin{axis}[
  xmin=0, xmax=1, ymin=0, ymax=1,
  xtick={0,1},
  ytick={0,1},
  xticklabels={\Large $0$,\Large$ 1$},
  yticklabels={\Large $0$,\Large $1$},
  enlargelimits=false,
  thick
 ]
 \addplot [
 samples=100,
 color=black,
 thick       
 ]
 {1-x};      
 \addplot [
 domain=0:0.6,
 samples=100, 
 color=black, 
 thick        
 ]
 {sqrt(1-x^2)/2};
 \node at (0.2,0.3){\Large $\Lambda_1^*$};
 \node at (0.2,0.6){\Large $\Lambda_2^*$};
\end{axis}
\end{tikzpicture}
\hspace{0.7cm}
\begin{tikzpicture}[xscale=0.6, yscale=0.6]
\begin{axis}[
  xmin=0, xmax=1, ymin=0, ymax=1,
  xtick={0,1},
  ytick={0,0.5991,1},
  xticklabels={\Large $0$,\Large $1$},
  yticklabels={\Large $0$,\Large $\eta^*$,\Large $1$},
  enlargelimits=false,
  thick
 ]
 \addplot [
 samples=100,
 color=black,
 thick           
 ]
 {0.5991};           
 \addplot [
 domain=0.343:1,
 samples=100, 
 color=black, 
 thick        
 ]
 {2/17 + sqrt( 4/(17*x^2) - 1/289) };  
 \node at (0.4,0.3){\Large $X_1^*$};
 \node at (0.4,0.8){\Large $X_2^*$};
 \node at (0.8,0.85){\Large $\xi^*(\eta)$};
\end{axis}
\end{tikzpicture}
}
\end{center}
\caption{The domains $\Lambda^*$ (left) and $X^*$ (right).} 
\label{fig:trnsDmn}
\end{figure}

Here the points $\eta^*$ and $\xi^*(\eta)$ are intersection points
where the paraboloid reaches the value $\rho=R$
\begin{equation*}
  \begin{aligned}
    \rho\big(1-\eta^*,\eta^*\big) &= R,\\
    \rho\big((1-\eta)\xi^*(\eta),\eta \xi^*(\eta)\big) &= R.   
  \end{aligned}
\end{equation*}
By construction, these quadratic equations have unique solutions in
$[0,1]$ which are smooth functions of $\xb_R,\xb_B,\yb_R,\yb_B$.
This shows that $I : S\times R \times B \to \R$ defined in
\eqref{def:I:lambda3} is a smooth function.

\section{Implementation and Numerical Results}\label{sec:implementation}
The pyramidal decomposition described in Section~\ref{sec:decomp} is
applicable for any types of finite convex polytopes and any
order of piecewise polynomial finite elements. However,
we have implemented the method only for pairs of triangles and
piecewise constant elements in space and time. The implementation can
be obtained from \texttt{github.com/johtausch/welle}.

We first give more detail how to perform
the stopping criterion and the apex selection for two faces of
the triangles $P_x$ and $P_y$. A triangle has seven faces (the triangle
itself, three edges and three vertices). For each of the 49 pairs of
faces $F=F_x\times F_y$ the following quantities are precomputed:\\
\begin{tabbing}
  \qquad \= $r_{max}$\quad \= maximal distance of points in $F_x$ and $F_y$,\\
  \> $r_{min}$ \> minimal distance distance of points in $F_x$ and $F_y$,\\
  \> $r_{minv}$ \> minimal distance of two vertices in $F_x$ and $F_y$,\\
  \> $\ab_{max} = (\xb_{max},\yb_{max})$\quad \; \= points where the maximum occurs,\\
  \> $\ab_{min} = (\xb_{min},\yb_{min})$ \> points where the minimum occurs,\\
  \>  $\ab_{minv} = (\xb_{minv},\yb_{minv})$\> vertices with minimal distance.\\
\end{tabbing}
These geometry calculations are elementary and inexpensive as they only involve linearly
constrained optimizations
with the quadratic function $(\xb,\yb) \mapsto \abs{\xb-\yb}^2$.
The maximum distance always occurs on a pair of vertices, whereas the
minimum can also occur on an edge or in the interior. Therefore it is
necessary to store the minimal vertex distance separately.
The stopping criterion in \ref{def:stop} is now satisfied if
either $r_{max}\leq\rho^*$ or $r_{min}\geq\rho^*$. For the apex
selection in \ref{def:apex} we adopt the strategy\\
\begin{center}
  If $r_{minv}\leq \rho^*$ choose $\ab_{minv}$, otherwise, choose $\ab_{min}$.
\end{center}
\vspace*{0.2cm}
This choice ensures that the convex hull of all apices will be
near and also prefers vertices over the minimizer $\ab_{min}$. This 
reduces the number of faces in $\faces{\ab}{F}$ if the minimizer is on
an edge or a face. 

After these geometry calculations have been performed, the integrals
for all $d$-values are computed for which the  function $g_d$ is non-zero.
Algorithm~\ref{algo:SetupMatrix} describes the complete process to
setup all matrices $A_d$.

\begin{algorithm}
  \caption{Algorithm to set up all nonzero matrix entries $A_{d,ij}$.}
\label{algo:SetupMatrix}
\begin{algorithmic}
\For{$i\in\{1,\dots,N_s\}$}
\For{$j\in\{1,\dots,N_s\}$}
\State Perform the geometry calculations for $P_i\times P_j$.
\State Set $d_{min}=\max\{ \lfloor r_{min}/h \rfloor,1\}$,
\State Set $d_{max}=\lceil r_{max}/h \rceil$
\State Set $B_{-1} = B_0 = 0$
\For{$d\in\{d_{min},\dots,d_{max}+1\}$}
\State Compute $B_d = \int_{P_i}\int_{P_j}  g_d(\abs{\xb-\yb}) d\yb
d\xb$ using pyramidal decomposition
\EndFor

\For{$d\in\{d_{min},\dots,d_{max}\}$}
\State Set $A_{d,ij} = -B_{d-1} + 2 B_d - B_{d+1}$
\EndFor
\EndFor 
\EndFor 
\end{algorithmic}
\end{algorithm}
Since the integers from $d_{min}$ and
$d_{min}$ for a given pair of triangles is a small subset of all $d\in
\{0,\dots,N_t\}$ the matrices $A_d$ will be sparse. Furthermore, if
\begin{equation*}
d > \frac{\mbox{diam}(\Gamma)}{h} + 1
\end{equation*}
then $A_d$ is the zero matrix.

\subsection{Convergence with respect to the quadrature order}\label{sec:quaderrConv}
We have tested the method for the case that the scatterer is the unit
sphere which has been triangulated in sequence of quasiuniform refinements consisting
of 48, 192, 768 and 3072 triangles. These meshes will be referred to as
$M_0$, $M_1$, $M_2$ and $M_3$, respectively. The iterated integrals in the
convex hulls either involve integrals over an interval and  integrals
over simplices. For the former we use Gauss Jacobi integration that is
rescaled to the length of the interval. For the latter we paramterize
the simplex with the standard simplex using an affine map. For the
quadrature on the standard simplex we use rules that are exact for
multvariable polynomials up to a specified degree that approximately
minimize the number of nodes. More details on these rules can be found
in~\cite{slobodkins-tausch23}.

The first experiment is to test the convergence 
with respect to increasing the degree of precision of the
quadrature rules. It is designed as follows: We let $N_t=20$ and
choose mesh $M_0$.
Algorithm~\ref{algo:SetupMatrix} is modified to compute only one
column of the matrices $A_d$, where all integrals and their
types are recorded. The process is repeated for quadrature orders
between 1 and 19 and the convergence of each integral is
observed. Figure~\ref{fig:convergence:type} records the maximal
errors within each of the six types of integrals described in
section~\ref{sec:analysis}. The reference value is the numerically
computed integral with
order 21. One can see that in all cases the rate of convergence is
exponential, which confirms that the functions $I(\zb_A,\zb_B)$ and
$I(\zb_S,\zb_R,\zb_B)$ are indeed smooth.
\begin{figure}
\begin{center}
\includegraphics[width=6cm]{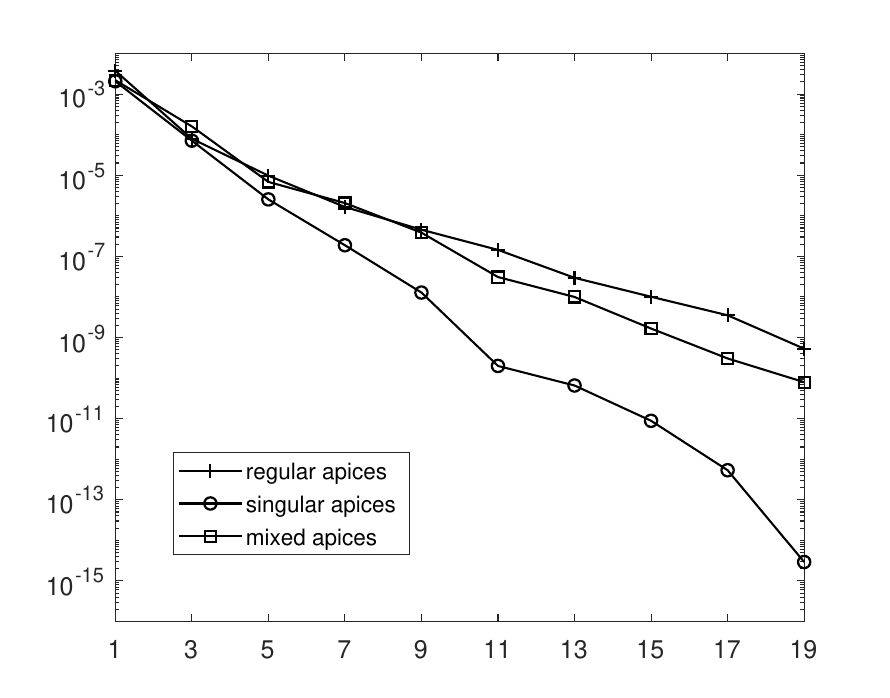}
\includegraphics[width=6cm]{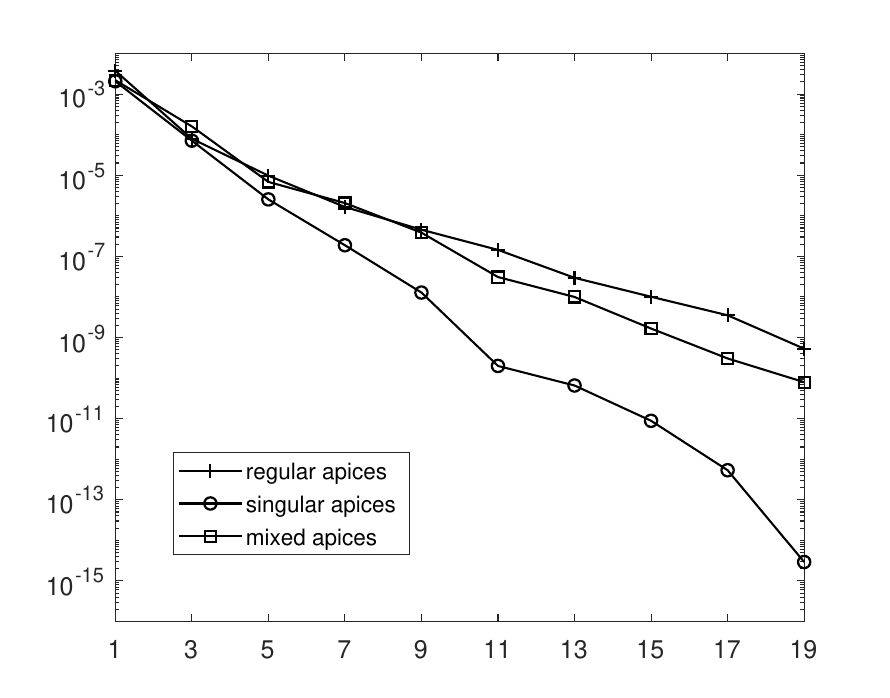}
\end{center}
\caption{Quadrature error for full regions (left) and truncated
  regions (right).} 
\label{fig:convergence:type}
\end{figure}

\subsection{Test problem with spherical scatterer}
This experiment illustrates the stability and accuracy for a spherical scatterer.
If the spatial dependence of the right hand side in \eqref{def:intgrEqn} is a
spherical harmonic in space, i.e., $f(\xb,t) = g(t)Y_n^m(\xb)$ then the solution is 
of the form $q(\xb,t) = \phi(t)Y_n^m(\xb)$. When $n=0$ the function
$\phi$ has a simple representation
\begin{equation*}
  \phi(t) = 2 \sum_{k=0}^{\lfloor t/2 \rfloor} g'(t-2k)
\end{equation*}
see, \cite{sauter-veit13}. We use this known solution with
$g(t) = t^4 \exp(-2t)/2$ as a reference
for the numerical solution obtained with the fully
discrete Galerkin scheme. The three parameters that determine the
accuracy are the
spatial and temporal meshwidths and the quadrature order. We use
piecewise constant elements and the triangulations $M_0$ to $M_3$
described in section~\ref{sec:quaderrConv}. For time, consider uniform
subdivisions of the interval $[0,20]$ in $400\cdot 2^k$ subintervals,
where $k\in\{0,\dots,4\}$. These meshes will be referred to by $T_k$.

\begin{table}[h]
\label{tab:errors}
\begin{center}
\begin{tabular}{ c|cccc}
 &  $M_0$ &   $M_1$ &  $M_2$ &  $M_3$ \\
  \hline
$T_0$ & 0.1965 &  0.1427 &  0.0466 &  0.1088\\
$T_1$ & 0.1988 &  0.1466 &  0.0434 &  0.1088\\
$T_2$ & 0.2000 &  0.1479 &  0.0427 &  0.0214\\
$T_3$ & 0.2003 &  0.1485 &  0.0425 &  0.0121\\
$T_4$ & 0.2005 &  0.1488 &  0.0424 &  0.0112
\end{tabular}
\end{center}
\caption{Discretization errors.}
\end{table}

Table~\ref{tab:errors} lists the errors of the solution of the fully
discrete scheme when the quadrature order is set to five. The
numbers shown are the quantities
\begin{equation*}
\max_{1\leq i\leq N_t}
  \abs{\frac{1}{\abs{\Gamma}}\int_\Gamma q_h(\xb,t_i) dS_\xb - \phi(t_i)}.
\end{equation*}
One can see that the error decreases as the spatial and temporal
meshwidth are refined. The rate of convergence appears somewhat faster than
the $O(h_t + h_s)$ rate that one would expect from piecewise constant
elements. This may be explained by the fact that the true solution for
a fixed time is constant and hence in the spatial ansatz space. However, the
spatial refinement still matters as the sphere is approximated by flat
triangles. The results seen in the table are also consistent with
unconditional stability. One should note
here that even though the retarded single layer potential is 
coercive in the sense of~\cite{bamberger-duong86}, the numerical analysis of the Galerkin
scheme is still not completely known. 

\begin{figure}
\begin{center}
\includegraphics[width=13cm]{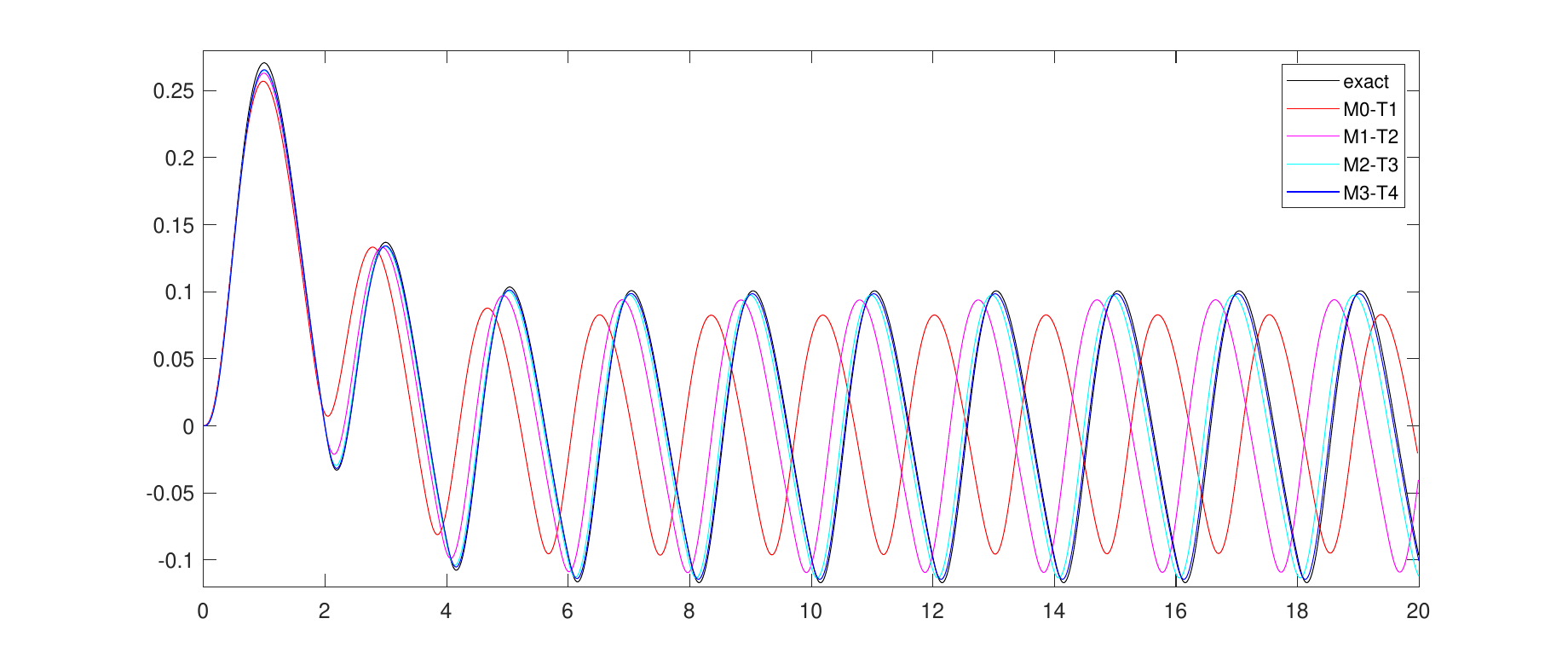}
\end{center}
\caption{Comparison of the exact density with the density obtained for
several mesh refinements.} 
\label{fig:solutions}
\end{figure}

Tables~\ref{tab:nonzeros} and~\ref{tab:nIntegrals} record the number
of nonzero entries in all $A_d$'s and the total number of convex hulls
integrals that had to be calculated to setup the matrices. The latter
number depends on the number of terms in the pyramidal decomposition.
This can change considerably, typical values are between 1 and 15, but
tends to be more frequently on the smaller side. The average is shown
in the brackets of table of~\ref{tab:nIntegrals}. This factor remains
mostly constant as the spatial and temporal meshes are refined.
Table~\ref{tab:cputime} displays the total CPU time to setup all $A_d$.
These numbers were obtained using a single processor thread. The
matrix entries could be easily evaluated in parallel, but this has not
been exploited in this implementation.

\begin{table}[h]
\label{tab:nonzeros}
\begin{center}
\begin{tabular}{ c|cccc}
  &  $M_0$ &   $M_1$ &  $M_2$ &  $M_3$ \\
  \hline
$T_0$ & 4.90e+04 & 4.41e+05 & 4.14e+06 & 4.26e+07\\
$T_1$ & 9.40e+04 & 8.12e+05 & 7.12e+06 & 4.26e+07\\
$T_2$ & 1.84e+05 & 1.55e+06 & 1.31e+07 & 1.14e+08\\
$T_3$ & 3.62e+05 & 3.03e+06 & 2.50e+07 & 2.10e+08\\
$T_4$ & 7.20e+05 & 5.97e+06 & 4.88e+07 & 4.01e+08
\end{tabular}
\end{center}
\caption{Total number of nonzeros in all $A_d$.}
\end{table}

\begin{table}[h]
\label{tab:nIntegrals}
\begin{center}
\begin{tabular}{ c|cccc}
 &  $M_0$ &   $M_1$ &  $M_2$ &  $M_3$ \\
  \hline
$T_0$ & 4.90e+04 (4.3) & 4.41e+05 (3.9) & 4.14e+06 (3.5) & 4.26e+07 (2.9)\\  
$T_1$ & 9.40e+04 (4.5) & 8.12e+05 (4.2) & 7.12e+06 (3.9) & 4.26e+07 (2.9)\\  
$T_2$ & 1.84e+05 (4.6) & 1.55e+06 (4.3) & 1.31e+07 (4.1) & 1.14e+08 (3.9)\\  
$T_3$ & 3.62e+05 (4.6) & 3.03e+06 (4.4) & 2.50e+07 (4.3) & 2.10e+08 (4.1)\\  
$T_4$ & 7.20e+05 (4.6) & 5.97e+06 (4.4) & 4.88e+07 (4.4) & 4.01e+08 (4.3)    
\end{tabular}
\end{center}
\caption{Total number of convex hull integrals with factors.}
\end{table}

\begin{table}[h]
\label{tab:cputime}
\begin{center}
\begin{tabular}{ c|cccc}
  &  $M_0$ &   $M_1$ &  $M_2$ &  $M_3$ \\
  \hline
$T_0$ & 4.72e-01 & 3.41e+00 & 2.47e+01 & 2.39e+02 \\
$T_1$ & 9.02e-01 & 8.25e+00 & 6.11e+01 & 2.37e+02 \\
$T_2$ & 1.81e+00 & 1.59e+01 & 1.30e+02 & 1.04e+03 \\
$T_3$ & 3.64e+00 & 3.10e+01 & 2.52e+02 & 1.96e+03 \\
$T_4$ & 6.84e+00 & 5.62e+01 & 4.86e+02 & 3.85e+03 
\end{tabular}
\end{center}
\caption{Total cpu time to setup the matrices $A_d$ in seconds.}
\end{table}

Figure~\ref{fig:quadSoln} explores the effect of the quadrature on the
numerical approximation of the density. The solution of the fully discrete is
calculated for the meshes $M_0$ and $M_1$ with 400 time steps. The
figure shows the differences
\begin{equation*}
\frac{1}{\abs{\Gamma}} \max_{1\leq i\leq N_t}
\abs{ \int_\Gamma q_{h,p}(\xb,t_i) dS_\xb
  - \int_\Gamma q_{h,19}(\xb,t_i) dS_\xb}.
\end{equation*}
where $q_{h,p}$ is the solution of the fully discrete scheme using
order-$p$ quadrature for the calculation of the matrix coefficients.
One can observe exponential decay when the mesh is fixed
and the quadrature error is increased.

\begin{figure}
\begin{center}
\includegraphics[width=6cm]{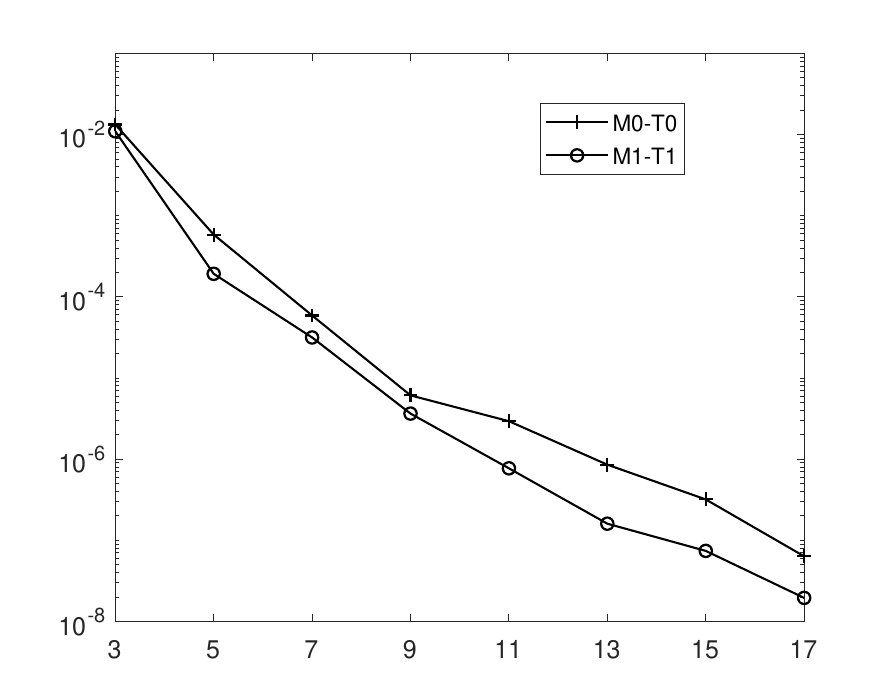}
\end{center}
\caption{Effect of the quadrature error on the numerical solution.} 
\label{fig:quadSoln}
\end{figure}

\bibliography{mrabbrev,big}

\begin{thebibliography}{10}

\bibitem{aimi-etal13}
A.~Aimi, M.~Diligenti, A.~Frangi, and C.~Guardasoni.
\newblock Neumann exterior wave propagation problems: computational aspects of
  {3D} energetic {G}alerkin {BEM}.
\newblock {\em Comput. Mech.}, 51:475–493, 2013.

\bibitem{aimi-etal09}
A.~Aimi, M.~Diligenti, C.~Guardasoni, I.~Mazzieri, and S.~Panizzi.
\newblock An energy approach to space–time {G}alerkin {BEM} for wave
  propagation problems.
\newblock {\em Int. J. Numer. Methods. Eng.}, 80(1196–1240), 2009.

\bibitem{aimi-etal10}
A.~Aimi, M.~Diligenti, and S.~Panizzi.
\newblock Energetic {G}alerkin {BEM} for wave propagation {N}eumann exterior
  problems.
\newblock {\em CMES}, 58(185-219), 2010.

\bibitem{bamberger-duong86}
A.\ Bamberger and T.~H. Duong.
\newblock Formulation variationnelle espace-temps pour le calcul par potentiel
  retard\'e la diffraction d une onde acoustique.
\newblock {\em Math. Methods Appl. Sci.}, 3(8):405–--435, 1986.

\bibitem{banz-etal16}
L.~Banz, H.~Gimperlein, Z.~Nezhi, and E.~P. Stephan.
\newblock Time domain {BEM} for sound radiation of tires.
\newblock {\em Comput Mech}, 58:45–57, 2016.

\bibitem{costabel04}
M.\ Costabel.
\newblock Time-dependent problems with the boundary integral equation method.
\newblock In E.~Stein, R.\ de~Borst, and T.\ Hughes, editors, {\em Encyclopedia
  of Computational Mathematics}. Wiley, 2004.

\bibitem{davies-duncan04}
P.~J. Davies and D.B. Duncan.
\newblock Stability and convergence of collocation schemes for retarded
  potential integral equations.
\newblock {\em SIAM J. Numer. Anal.}, 42(3):1167–1188, 2004.

\bibitem{davies-duncan14}
P.~J. Davies and D.B. Duncan.
\newblock Convolution spline approximations for time domain boundary integral
  equations.
\newblock {\em J. Integr. Eqns. Appl}, 26(3):369--410, 2014.

\bibitem{gimperlein-etal17}
H.~Gimperlein, M.~Maischak, and E.P. Stephan.
\newblock Adaptive time domain boundary element methods with engineering
  applications.
\newblock {\em J. Integral Equations Appl.}, 29(1), 75--105 2017.

\bibitem{grunbaum03}
B.\ Gr\"unbaum.
\newblock {\em Convex Polytopes}, volume 221 of {\em Graduate Texts in
  Mathematics}.
\newblock Springer, 2003.

\bibitem{joly-rodríguez17}
P.~Joly and J.~Rodríguez.
\newblock Mathematical aspects of variational boundary integral equations for
  time dependent wave propagation.
\newblock {\em J. Integral Equations Appl.}, 29(1):137–187, 2017.

\bibitem{khoromskij-etal11}
B.~Khoromskij, S.~Sauter, and A.~Veit.
\newblock Fast quadrature techniques for retarded potentials based on {TT/QTT}
  tensor approximation.
\newblock {\em Comput. Meth. Appl. Math.}, 11(3):342–362, 2011.

\bibitem{kress-sauter08}
W.~Kress and S.~Sauter.
\newblock Numerical treatment of retarded boundary integral equations by sparse
  panel clustering.
\newblock {\em IMA J. Numer. Anal.}, 28:162--185, 2008.

\bibitem{lubich94}
C.\ Lubich.
\newblock On the multistep time discretization of linear initial–boundary
  value problems and their boundary integral equations.
\newblock {\em Numer. Math.}, 67(3):365–389, 1994.

\bibitem{ostermann10}
E.~Ostermann.
\newblock {\em Numerical Methods for Space-Time Variational Formulations of
  Retarded Potential Boundary Integral Equations}.
\newblock PhD thesis, Wilhelm Leibniz Universit\"at, Hannover, 2010.

\bibitem{polz-schanz19}
D.~P\"olz and M.~Schanz.
\newblock Space-time discretized retarded potential boundary integral
  operators: Quadrature for collocation methods.
\newblock {\em SIAM J. Sci. Comput.}, 41(6):A3860–A3886, 2019.

\bibitem{polz-schanz21}
D.~P\"olz and M.~Schanz.
\newblock On the space-time discretization of variational retarded potential
  boundary integral equations.
\newblock {\em Comput. Math. Appl.}, 99:195--210, 2021.

\bibitem{sauter-schwab11}
S.\ Sauter and C.\ Schwab.
\newblock {\em Boundary Element Methods}.
\newblock Springer, 2011.

\bibitem{sauter-veit13}
S.~Sauter and A.~Veit.
\newblock A {G}alerkin method for retarded boundary integral equations with
  smooth and compactly supported temporal basis functions.
\newblock {\em Numer. Math.}, 123:145–176, 2013.

\bibitem{sayas16}
F.-J.\ Sayas.
\newblock {\em Retarded Potentials and Time Domain Boundary Integral
  Equations}.
\newblock Number~50 in Springer Series in Computational Mathematics. Springer,
  2016.

\bibitem{schanz-antes97}
M.~Schanz and H.~Antes.
\newblock Application of operational quadrature methods in time domain boundary
  element methods.
\newblock {\em Meccanica}, 32(3):179--186, 1997.

\bibitem{schneider-etal26}
S.~Schneider, C.~\"Ozdemir, H.~Gimperlein, K.~Urban, and B.~Graf.
\newblock A stable boundary element method for reliable long-time industrial
  sound emission.
\newblock {\em Comp. Mech.}, 2026.

\bibitem{slobodkins-tausch23}
A.~Slobodkins and J.~Tausch.
\newblock A node elimination algorithm for cubature of high-dimensional
  polytopes.
\newblock {\em Computers {\&} Math. Appl.}, 144:229--236, 2023.

\bibitem{tausch26}
J.~Tausch.
\newblock Quadrature for singular integrals over convex polytopes.
\newblock {\em arXiv preprint}, 2025.
\newblock 2511.13974.

\bibitem{ziegler95}
G.~Ziegler.
\newblock {\em Lectures on Polytopes}, volume 152 of {\em Graduate Texts in
  Mathematics}.
\newblock Springer, 1995.

\end{thebibliography}
\bibliographystyle{plain}

\end{document}